\documentclass[12 pt]{amsart}
\usepackage{setspace, amsmath, amsthm, amssymb, amsfonts, amscd, epic, graphicx, ulem, dsfont}
\usepackage[T1]{fontenc}
\usepackage{multirow}
\usepackage{bbm}
\usepackage{enumerate}

\makeatletter \@namedef{subjclassname@2010}{
  \textup{2010} Mathematics Subject Classification}
\makeatother

\newtheorem{lem}{Lemma}
\newtheorem{thm}{Theorem}
\newtheorem{pro}{Proposition}

\theoremstyle{remark}
\newtheorem*{rema}{Remark}

\theoremstyle{definition}

\newcommand{\C}{\mathbb{C}}

\begin{document}

\title[Exponentials of Normal Operators]{Exponentials of Normal Operators and Commutativity of Operators: A New Approach.}
\author[M. H. MORTAD]{MOHAMMED HICHEM MORTAD}

\dedicatory{}
\date{}
\thanks{}



\address{
D\'{e}partement de Math\'{e}matiques, Universit\'{e} d'Oran
(Es-Senia), B.P. 1524, El Menouar, Oran 31000, Algeria.\newline {\bf
Mailing address}: \newline Dr Mohammed Hichem Mortad \newline BP
7085 Es-Seddikia\newline Oran \newline 31013 \newline Algeria}

\email{mhmortad@gmail.com, mortad@univ-oran.dz.}

\begin{abstract}
We present a new approach to the question of when the commutativity
of operator exponentials implies that of the operators.  This is
proved in the setting of bounded normal operators on a complex
Hilbert space. The proofs are based on some similarities results by
Berberian and Embry as well as the celebrated Fuglede theorem.
\end{abstract}

\maketitle

\section{Introduction}
Let $A$ and $B$ be two linear bounded operators on a \textit{Banach}
space $H$. The way of defining $e^A$, the exponential of $A$, is
known at an undergraduate level.  The functions $\sinh A$, $\cosh A$
can be defined similarly. It is also an easy exercise to show that
$AB=BA$ implies $e^Ae^B=e^Be^A$.  While the converse is not always
true, it is, however, true under the hypothesis that $A$ and $B$
self-adjoint on a $\C$-Hilbert space. This result is stated in the
following theorem (for the readers convenience, a proof may be found
in \cite{wermuth.1997.pams})
\begin{thm}\label{e^Ae^B=e^Be^A iff A=B A,B self-adjoint}
Let $A$ and $B$ be two self-adjoint operators defined on a Hilbert
space. Then
\[e^Ae^B=e^Be^A \Longleftrightarrow AB=BA.\]
\end{thm}

There have been several attempts to prove the previous theorem for
non self-adjoint operators using the $2\pi i$-congruence-free
hypothesis (see eg.
\cite{Fotios-Plaiogannis,schmoeger1999,schmoeger2000,schmoeger2001,wermuth.1997.pams}).
See also \cite{Bourgeois} for some low dimensions results without
the $2\pi i$-congruence-free hypothesis.

In this paper, we present a different approach to this problem using
results about similarities due to Berberian
\cite{Berberian-adjoint-similar} and also to Embry \cite{Embry
Simil.}.  The main question asked here is under which assumptions we
have

\[e^Ae^B=e^Be^A\Longrightarrow AB=BA\]
for normal $A$ and $B$?

Another result obtained in this article is about sufficient
conditions implying $AB=BC$ given that $e^Ae^B=e^Be^C$, where $A$,
$B$ and $C$ are self-adjoint operators.

All operators considered in this paper are assumed to be bounded and
defined on a separable complex Hilbert space. The notions of normal,
self-adjoint and unitary operators are defined in the usual fashion.
So is the notion of the spectrum (with the usual notation $\sigma$).
It is, however, more convenient to recall the notion of a cramped
operator. A \textit{unitary} operator $U$ is said to be
\textit{cramped} if its spectrum is completely contained on some
open semi-circle (of the unit circle), that is
\[\sigma(U)\subseteq \{e^{i t}:~\alpha<t<\alpha+\pi\}.\]

While we assume the reader is familiar with other notions and
results on bounded operators (some standard references are
\cite{Con,RUD}), we will be recalling on several occasions some
results that might not be known to some readers. One of them is the
following result (first established in \cite{BeckPut}).

\begin{thm}\label{Berberian}[Berberian, \cite{Berberian-adjoint-similar}]
Let $U$ be a cramped operator and let $X$ be a bounded operator such
that $UXU^*=X^*$. Then $X$ is self-adjoint.
\end{thm}

\section{Main Results}

The following lemma is fundamental to prove our results. Its proof
follows from the holomorphic functional calculus.

\begin{lem}\label{Lemma A^iB^i=B^iA^i}
Let $A$ and $B$ be two commuting normal operators, on a Hilbert
space, having spectra contained in simply connected regions not
containing 0 . Then
\[A^iB^i=B^iA^i\]
where $i=\sqrt{-1}$ is the usual complex number.
\end{lem}

\begin{lem}\label{e^S^i=e^iS}
Let $A$ be a self-adjoint operator such that $\sigma(A)\subset
(0,\pi)$. Then
\[(e^{iA})^i=e^{-A}.\]
\end{lem}

\begin{proof}
The proof follows from the functional calculus.
\end{proof}

Before stating and proving the main theorem, we first give an
intermediate result.

\begin{pro}\label{e^Se^N S self-adjoint, N normal}
Let $N$ be a normal operator with cartesian decomposition $A+iB$.
Let $S$ be a self-adjoint operator.  If $\sigma(B)\subset (0,\pi)$,
then
\[e^Se^N=e^Ne^S\Longrightarrow SN=NS.\]
\end{pro}

\begin{proof}
Let $N=A+iB$ where $A$ and $B$ are two commuting self-adjoint
operators.  Hence $e^Ae^{iB}=e^{iB}e^A$. Consequently,
\begin{align*}
e^Se^N=e^Ne^S&\Longleftrightarrow e^Se^Ae^{iB}=e^Ae^{iB}e^S\\
&\Longleftrightarrow e^Se^Ae^{iB}=e^{iB}e^Ae^S\\
&\Longleftrightarrow e^Se^Ae^{iB}=e^{iB}(e^Se^A)^*
\end{align*}
Since $B$ is self-adjoint, $e^{iB}$ is unitary. It is also cramped
by the spectral hypothesis on $B$. Now, Theorem \ref{Berberian}
implies that $e^Se^A$ is self-adjoint, i.e.
\[e^Se^A=e^Ae^S.\]
Theorem \ref{e^Ae^B=e^Be^A iff A=B A,B self-adjoint} then gives us
$AS=SA$.

It only remains to show that $BS=SB$. Since $e^Se^A=e^Ae^S$, we
immediately obtain
\[e^Se^N=e^Ne^S \Longrightarrow e^Se^Ae^{iB}=e^Ae^{iB}e^S \text{ or } e^Ae^Se^{iB}=e^Ae^{iB}e^S\]
and so
\[e^Se^{iB}=e^{iB}e^S\]
by the invertibility of $e^A$.

Using Lemmas \ref{e^Ae^B=e^Be^A iff A=B A,B self-adjoint} \&
\ref{e^S^i=e^iS} we immediately see that
\[e^Se^{-B}=e^{-B}e^S.\]
Theorem \ref{e^Ae^B=e^Be^A iff A=B A,B self-adjoint} yields $BS=SB$
and thus
\[SN=S(A+iB)=(A+iB)S=NS.\]
The proof of the proposition is complete.
\end{proof}

\begin{rema}We could have bypassed Berberian's result by alternatively using
some of Embry's results (see \cite{Embry Simil.}).

Similar results can be obtained too by a result due to I. H. Sheth
\cite{Sheth-Hyponormal}. They all more or less deduce the
self-adjointness of an operator $N$ from the operational equation
$AN=N^*A$ (under obviously extra conditions on $N$ and/or on $A$).

In the proof of the previous proposition, $N$ is a product of
self-adjoint operators. Then a result by the author (which appeared
in \cite{MHM1}) can be applied too.
\end{rema}

Now, we state and prove the main theorem in this paper, namely
\begin{thm}\label{e^Me^N N,M normal operators}
Let $N$ and $M$ be two normal operators with cartesian
decompositions $A+iB$ and $C+iD$ respectively. If
$\sigma(B),\sigma(D)\subset (0,\pi)$, then
\[e^Me^N=e^Ne^M\Longrightarrow MN=NM.\]
\end{thm}

\begin{proof}

We have

\[e^Me^N=e^Ne^M \Longrightarrow e^{M^*}e^N=e^{N}e^{M^*}\]
by the Fuglede theorem since $e^M$ is normal. Hence by using again
the normality of $M$
\[ e^{M^*}e^Me^N=e^{M^*}e^Ne^M\Longrightarrow e^{M^*}e^Me^N=e^Ne^{M^*}e^M\]

or

\[e^{M^*+M}e^N=e^Ne^{M^*+M}\]

Since $M^*+M$ is self-adjoint, Proposition \ref{e^Se^N S
self-adjoint, N normal} applies and gives

\[(M^*+M)N=N(M^*+M) \text{ or just } CN=NC.\]

The previous implies that $N^*C=CN^*$ and thus $(N+N^*)C=C(N+N^*)$.
Therefore, we have

\[AC=CA \text{ and hence } BC=CB.\]

Doing the same work for $N$ in lieu of $M$, very similar arguments
and Proposition \ref{e^Se^N S self-adjoint, N normal} all yield
\[AM=MA \text{ and hence } AD=DA.\]

To prove the remaining bit, we go back to the equation
$e^Ne^M=e^Me^N$. Then by the commutativity of $B$ and $C$ and by
that of $A$ and $D$, we obtain
\[e^Ae^{iB}e^Ce^{iD}=e^Ce^{iD}e^Ae^{iB}\Longleftrightarrow e^Ae^Ce^{iB}e^{iD}=e^Ce^Ae^{iD}e^{iB}.\]
Since $A$ and $C$ commute and since $e^Ae^C$ is invertible, we are
left with
\[e^{iB}e^{iD}=e^{iD}e^{iB}.\]
Lemmas \ref{Lemma A^iB^i=B^iA^i} \& \ref{e^S^i=e^iS} yield
\[e^{-B}e^{-D}=e^{-D}e^{-B}\]
which leads to $BD=DB$. Hence $BM=MB$.

Finally, we have
\[NM=(A+iB)M=AM+iBM=MA+iMB=M(A+iB)=MN,\]
completing the proof.
\end{proof}

We finish this paper by a result on self-adjoint operators which
generalizes Theorem \ref{e^Ae^B=e^Be^A iff A=B A,B self-adjoint} to
the case of three operators. We have
\begin{thm}
Let $\mathcal{H}$ be a $\C$-Hilbert space.  Let $A$, $B$ and $C$ be
all self-adjoint operators on $\mathcal{H}$.  If
\[\left\{\begin{array}{c}
    \cosh Ae^B=e^B\cosh A, \\
     \sinh A e^C=e^B\sinh A,\\
     e^C\cosh A=\cosh Ae^C,\\
e^C\sinh A=\sinh Ae^B,
  \end{array}\right.
\]
then
\[AC=BA\]
\end{thm}

\begin{proof}
Define on $\mathcal{H}\oplus \mathcal{H}$ the operators
\[\tilde{A}=\left(
              \begin{array}{cc}
                0 & A \\
                A & 0 \\
              \end{array}
            \right)
\text{ and } \tilde{B}=\left(
                         \begin{array}{cc}
                           B & 0 \\
                           0 & C \\
                         \end{array}
                       \right).
\]
One has
\[\tilde{A}^2=\left(
              \begin{array}{cc}
                A^2 & 0 \\
                0 & A^2 \\
              \end{array}
            \right),~
            \tilde{A}^3=\left(
              \begin{array}{cc}
                0 & A^3 \\
                A^3 & 0 \\
              \end{array}
            \right),\cdots\]
Hence
\[e^{\tilde{A}}=\left(
                  \begin{array}{cc}
                   I+\frac{A^2}{2!}+\frac{A^4}{4!}+\cdots  & A+\frac{A^3}{3!}+\frac{A^5}{5!}+\cdots  \\
                   A+\frac{A^3}{3!}+\frac{A^5}{5!}+\cdots  & I+\frac{A^2}{2!}+\frac{A^4}{4!}+\cdots \\
                  \end{array}
                \right)
=\left(
   \begin{array}{cc}
     \cosh A & \sinh A \\
     \sinh A & \cosh A \\
   \end{array}
 \right).
\]
Similarly, we can find that
\[e^{\tilde{B}}=\left(
                  \begin{array}{cc}
                    e^B & 0 \\
                    0 & e^C \\
                  \end{array}
                \right) .\] The hypotheses of the theorem guarantee
that $e^{\tilde{A}}e^{\tilde{B}}=e^{\tilde{B}}e^{\tilde{A}}$ and
since $\tilde{A}$ and $\tilde{B}$ are both self-adjoint, Theorem
\ref{e^Ae^B=e^Be^A iff A=B A,B self-adjoint} then implies that
$\tilde{A}\tilde{B}=\tilde{B}\tilde{A}$.

Examining the entries of the matrices $\tilde{A}\tilde{B}$ and
$\tilde{B}\tilde{A}$, we see that $AC=BA$, establishing the result
\end{proof}

\section*{Acknowledgement}
The author wishes to thank the referee for his/her remarks.

\bibliographystyle{amsplain}

\end{document}